\documentclass{amsart}[12pt]
\usepackage{float}
\usepackage{lmodern}
\usepackage{graphicx}
\usepackage{amssymb, mathrsfs, amsmath,amssymb}
\usepackage{fullpage}
\theoremstyle{plain}

\newtheorem{corollary}{Corollary}

\newtheorem{definition}{Definition}

\newtheorem{proposition}{Proposition}
\newtheorem{remark}{Remark}

\numberwithin{equation}{section}


\begin{document}
\title{Analysis of generalized negative binomial distributions attached to hyperbolic Landau levels}
\author[H. Chhaiba]{Hassan Chhaiba}
\address{Department of Mathematics \\ Faculty of Sciences, Ibn Tofail University \\ P.O. Box 133, Kénitra \\ Morocco}
\email{chhaiba.hassan@gmail.com}

\author[N. Demni]{Nizar Demni}
\address{IRMAR, Universit\'e de Rennes 1\\ Campus de
Beaulieu\\ 35042 Rennes cedex\\ France}
\email{nizar.demni@univ-rennes1.fr}

\author[Z. Mouayn]{Zouhair Mouayn}
\address{Department of Mathematics\\ Faculty of Sciences and Technics (M'Ghila)\\ Sultan Moulay Slimane \\ PO. Box 523, B\'eni Mellal\\ Morocco}
\email{mouayn@fstbm.ac.ma}

\maketitle

\begin{abstract}
To each hyperbolic Landau level of the Poincar\'e disc is attached a generalized negative binomial distribution. In this paper, we compute the moment generating function of this distribution and supply its decomposition as a perturbation of the negative binomial distribution by a finitely-supported measure. Using the Mandel parameter, we also discuss the nonclassical nature of the associated coherent states. Next, we determine the L\'evy-Kintchine decomposition its characteristic function when the latter does not vanish and deduce that it is quasi-infinitely divisible except for the lowest hyperbolic Landau level corresponding to the negative binomial distribution. By considering the total variation of the obtained quasi-L\'evy measure, we introduce a new infinitely-divisible distribution for which we derive the characteristic function.    
 \end{abstract}

\section{Introduction}
The negative binomial coherent states (NBCS) were constructed in \cite{Fu-Sas} in a standard fashion from the negative binomial distribution (NBD). They interpolate between thermal and coherent states 
(\cite{GJT}) and reduce also to Susskind-Glogower phase states in a particular limiting regime (\cite{Fu-Sas}). As explained in \cite{Bar}, the NBD counts the total number of present photons given the number of detected ones in opposite to the binomial distribution. It also plays a key role in the description of random multiplicities in the process of multiple production of electrically-charged species (\cite{Gio}). 
From a geometrical point of view, NBCS exhibit a $SU(1,1)$-symmetry and their coefficients are eigenfunctions of a second-order differential operator in the Poincar\'e disc corresponding to the lowest eigenvalue (\cite{Mou}). This operator can be mapped using a gauge transformation to the Schr\"odinger operator in the disc with a constant magnetic field and the latter is the image by a Cayley transform of the Maass operator studied in \cite{Com} (see \cite{Mo} and references therein for further details). 
For a strong enough intensity of the magnetic field ($\nu > 1/2$), the spectrum admits a discrete (non degenerate) part named hyperbolic Landau levels given by 
the sequence $\epsilon_m^{(\nu)} = 4m\left( 2\nu -m-1\right), m=0,1,...,[\nu -(1/2)]$, where $[a]$ denotes the greater integer less than $a$. 

In \cite{Mou}, generalized negative binomial coherent states (GNBCS) were attached to eigenspaces corresponding to hyperbolic Landau levels $\epsilon_m^{(\nu)}, m \geq 0$ and give rise for any fixed $m$ to the following random variable $X_m$:
\begin{equation*}
\mathbb{P}(X_{m}=j) =\gamma _{j}^{(\nu, m, R)} \left(\tau R^{-2}\right)^{\left| m-j\right|}(1-\tau R^{-2})^{2\nu R^2-2m}
\left(P_{\frac{1}{2}\left( m+j-\left| m-j\right| \right)}^{\left( \left| m-j\right| ,2\nu R^2-2m-1\right) }\left( 1-2\tau R^{-2}\right) \right) ^{2},\text{ }j\geq 0.
\end{equation*}
In this equation, $P_{n}^{\left( \alpha ,\beta \right) }$ stands for the $n$-th Jacobi polynomial, $z$ belongs to the Poincar\'e disc $\mathbb{D}_R$ of radius $R > 0$, $\tau =z\overline{z}$ is the intensity of the light and $\gamma _{j}^{(\nu, m, R)} > 0$ is a normalizing constant. Since $X_0$ is a distributed as a NBD, we refer below to the distribution of $X_{m}$ as the generalized negative binomial distribution (GNBD).

In this paper, we derive a closed formula for the moment generating function of $X_{m}$ from which we deduce the mean and the variance. Doing so enables us to discuss the nonclassical nature of the corresponding statistics through the Mandel parameter. 
For $m=0,$ it is already known that the NBD exhibits super-Poissonian statistics describing a thermal light field with a bunched photon spacing (\cite{GJT}). However, the situation changes drastically for higher hyperbolic Landau levels $\epsilon_m^{(\nu)}, m \geq 1$ and the statistics can be sub-Poissonian (anti bunched), Poissonian (coherent) or super-Poissonian (bunched) depending on the light intensity $\tau$. Moreover, the Poissonian statistics occur when $\tau$ lies on a circle whose radius increases as $m$ does while keeping $R$ and $\nu$ fixed.  
 
We also study here some probabilistic aspects of the random variable $X_{m}$. For instance, we give for any $m \geq 1$ the decomposition of the GNBD as a perturbation of the NBD corresponding to $m=0$ by a finitely-supported measure. Besides, the mass function of the perturbation is expressed through a ${}_{3}F_{2}$ hypergeometric polynomial in the variable $\tau$. Next, we derive the L\'evy-Kintchine decomposition of the characteristic function of $X_m$ when the latter does not vanish and deduce that the GNBD is quasi-infinitely divisible unless $m=0$. It is remarkable that this quasi-infinite divisibility goes in parallel with the emergence of the zone where the statistics of $X_m$ are anti-bunching. 
Finally, by considering the total variation of the obtained quasi-L\'evy measure, we introduce a new infinitely-divisible distribution or equivalently a new L\'evy process.  

The paper is organized as follows. In section 2, we collect the definitions and some properties of various special functions which occur in the remainder of the paper. We also recall in the same section the 
coherent states formalism we will be using as well as basic facts on generalized Bergman spaces on the disk. Section 3 is devoted to the introduction of the GNBD together with the derivation of a closed expression for its moment generating function. There, we also write down the atomic decomposition of the GNBD and discuss the nonclassical nature of its statistics of $X_{m}$. In the last section, we derive the L\'{e}vy-Khintchine representation of the characteristic function of $X_m$ when the latter does not vanish and introduce the infinitely-divisible distribution whose L\'evy measure is the total variation of the quasi-L\'evy measure of the GNBD.

\section{Special functions and coherent states of generalized Bergman spaces in the unit disc} 
\subsection{Special functions}
We start with the Gamma function defined for $x > 0$ by (\cite{AAR}, p.6):
\begin{equation*}
\Gamma(x) = \int_0^{\infty} e^{-u}u^{x-1}du.
\end{equation*}
This function satisfies the Legendre duplication formula (\cite{AAR}, p.22):
\begin{equation}\label{Leg}
\sqrt{\pi}\Gamma(2x+1) = 2^{2x} \Gamma\left(x+\frac{1}{2}\right)\Gamma(x+1). 
\end{equation}
Next, let $k \geq 1$ be a positive integer and recall the Pochhammer symbol 
\begin{equation*}
(x)_k = (x+k-1)\dots (x+1)x
\end{equation*}
with the convention $(x)_0 := 1$. Note that 
\begin{equation*}
(x)_k = \frac{\Gamma(x+k)}{\Gamma(x)}, \,\, x > 0, 
\end{equation*}
while
\begin{eqnarray}
(-n)_k & = & \frac{(-1)^kn!}{(n-k)!}, \,\, k \leq n, \\ 
 &= & 0, \quad \quad \quad \,\, k > n \nonumber.
\end{eqnarray}
Now, let $a, b > -1$ be real numbers, then the Jacobi polynomial of parameters $a, b$ is defined by the expansion (which follows readily from Rodriguez formula, see \cite{AAR}, p.99): 
\begin{equation}\label{DefJac}
P_n^{(a,b)}(x) := \sum_{k=0}^n \binom{n+a}{k} \binom{n+b}{n-k} \left(\frac{x-1}{2}\right)^{n-k} \left(\frac{x+1}{2}\right)^{k}.
\end{equation}
When $a,b > -1$, the Jacobi polynomials $(P_n^{(a,b)})_n$ are orthogonal with respect to the Beta weight 
\begin{equation*}
(1-x)^a(1+x)^b{\bf 1}_{[-1,1]}(x)
\end{equation*}
and as such their zeros are simple and lie in $(-1,1)$ (\cite{AAR}, p.253). For sake of simplicity, denote
\begin{equation*}
-1 < x_1^{(n)} < \dots < x_n^{(n)} < 1 
\end{equation*} 
the increasingly-ordered zeros of $P_n^{(a,b)}$ without any reference to their dependence on the parameters $(a,b)$. Then
\begin{equation}\label{Fact}
P_n^{(a,b)}(x) = \frac{(n+a+b+1)_n}{2^nn!} \prod_{i=1}^n(x-x_i^{(n)}).
\end{equation}
Finally, we recall the definition of the hypergeometric series: for any positive integers $p, q$, 
\begin{equation*}
{}_pF_q(a_1, \dots, a_p; b_1, \dots; b_q; x) := \sum_{k \geq 0}\frac{(a_1)_k \dots (a_p)_k}{(b_1)_q\dots(b_q)_k} \frac{x^k}{k!}
\end{equation*}
whenever the series converges (\cite{AAR}, p.61-62). In this definition, $(a_i, 1 \leq i \leq p)$ are real numbers while $b_i \in \mathbb{R} \setminus \mathbb{N}$ for any $1 \leq i \leq q$. Note that if $a_i = -n$ for some $1 \leq i \leq p$ then the series terminates and we end up with a hypergeometric polynomial. For instance, the representation 
\begin{equation}\label{RepJac}
P_n^{(a,b)}(x) = \frac{(a+1)_n}{n!}{}_2F_1\left(-n, n+a+b+1, a+1; \frac{1-x}{2}\right)
\end{equation}
provides another definition of the Jacobi polynomial when $a,b > -1$. We shall also need the definition of Laguerre polynomials 
\begin{equation*}
L_n^{(a)}(x) = \frac{(a+1)_n}{n!}{}_1F_1(-n, a+1, x) = \frac{(a+1)_n}{n!} \sum_{k=0}^n \binom{n}{k}\frac{(-x)^k}{(a+1)_k k!}, \,\, a > -1.
\end{equation*}

\subsection{Coherent states}

\noindent In this paragraph, we recall from \cite{Gaz} (p.72-75) the formalism based on (generalized) coherent states leading to GNBD. 
 Let $(X, \mu)$ be a measurable space and denote $L^{2}(X,d\mu)$\ the space of $\mu$-square integrable functions on $X$. Let $\mathcal{A} \subset L^{2}(X,d\mu )$ be a closed subspace of infinite dimension with an orthonormal basis $\left\{ \Phi _{j}\right\} _{j=0}^{\infty }$ and let $(\mathcal{H}, \langle \,\mid \,\rangle)$ be a infinite-dimensional separable Hilbert space equipped with an orthonormal basis $\left\{ \phi _{j}\right\} _{j=0}^{\infty }$. Then the coherent states $\left\{ \mid x>\right\} _{x\in X}$ in $\mathcal{H}$ are defined by
\begin{equation}\label{CohSta}
\mid x>:=\left( \mathcal{N}\left( x\right) \right) ^{-\frac{1}{2}} \sum_{j=0}^{+\infty }\Phi _{j}\left( x\right) \mid \phi _{j}>\quad 
\end{equation}
where 
\begin{equation}
\mathcal{N}\left( x\right) =\sum_{j=0}^{+\infty }\Phi _{j}\left( x\right) 
\overline{\Phi _{j}\left( x\right) }.  \tag{2.1.2}
\end{equation}
They obey the normalization condition 
\begin{equation}
\left\langle x\mid x\right\rangle _{\mathcal{H}}=1  \tag{2.1.3}
\end{equation}
and provide the following resolution of the identity operator:
\begin{equation}\label{Res}
\mathbf{1}_{\mathcal{H}}=\int\limits_{X}\mid x><x\mid \mathcal{N}\left(
x\right) d\mu \left( x\right)  \tag{2.1.4}
\end{equation}
where $\mid x><x\mid$ is the Dirac's bra-ket notation for the rank-one operator $\varphi \mapsto \langle x \mid \varphi \rangle  x$. Note in passing that the choice of the Hilbert space $\mathcal{H}$ defines a quantization of the space $X$ by the coherent states defined by \eqref{CohSta} via the inclusion map $X\ni x\mapsto \mid x>\in \mathcal{H}$. In this respect, the property \eqref{Res} bridges between classical and quantum mechanics. 
\subsection{Generalized Bergman spaces}
Now, take $X = \mathbb{D}_1 = \left\{ z\in \mathbb{C},\left| z\right| <1\right\} = \mathbb{D}$ to be the unit disk endowed with the measure  
\begin{equation}
\mu_{\nu} \left(dz\right) := \left(1-z\overline{z}\right) ^{2\nu-2} \lambda(dz), \,\, \nu \geq 0,  
\end{equation}
$\lambda(dz)$ being the Lebesgue measure in $\mathbb{C}$. Note that $\mu_0$ is the volume element of $\mathbb{D}$ when the latter is equipped with its usual K\"{a}hler metric $ds^{2}=-\partial \overline{\partial }\ln\left( 1-z\overline{z}\right) dz\otimes d\overline{z}$. Consider the second order differential operator 
\begin{equation}
\Delta _{\nu }:=-4\left( 1-z\overline{z}\right) \left( \left( 1-z\overline{z}%
\right) \frac{\partial ^{2}}{\partial z\partial \overline{z}}-2\nu \overline{%
z}\frac{\partial }{\partial \overline{z}}\right) .  \tag{2.2.5}
\end{equation}
This is an elliptic and densely defined operator on the Hilbert space $L^{2}\left( \mathbb{D}, \mu_{\nu}\right)$, and admits a unique self-adjoint realization that we also denote $\Delta_{\nu }$ (\cite{Mo}). Moreover, its spectrum consists of $(i)$%
\textit{\ }a continuous part $\left[ 1,+\infty \right[ $ corresponding to \textit{scattering states}, $(ii)$ a finite number of eigenvalues (\textit{hyperbolic Landau levels}) of the form 
\begin{equation}\label{HLL}
\epsilon _{m}^{(\nu)}:=4m\left(2\nu -m-1\right) ,m=0,1,2,\cdots ,\left[ \nu -%
\frac{1}{2}\right] 
\end{equation}
with infinite multiplicities, provided that $2\nu >1$. To each eigenvalue $\epsilon_m^{(\nu)}$ corresponds eigenfunctions which are referred to as \textit{bound states}. Let
\begin{equation*}
\mathcal{A}_{\nu ,m}^{2}(\mathbb{D}):=\left\{ \Phi :\mathbb{D}\mathbf{\rightarrow }\mathbb{C}, \, \Phi \in L^{2}\left(\mathbb{D}, \mu_{\nu}\right), \, \, \Delta _{\nu }\Phi =\epsilon _{m}^{(\nu)}\Phi \right\}
\end{equation*}
be the eigenspace associated with $\epsilon _{m}^{\nu }$. Then an orthonormal basis of $\mathcal{A}_{\nu ,m}^{2}\left( \mathbb{D}\right)$ is given by
\begin{multline*}
\Phi _{k}^{(\nu ,m)}\left( z\right) =\left( -1\right) ^{k}\left( \frac{2\left(\nu -m\right) -1}{\pi }\right) ^{1/2} 
\left( \frac{k!\Gamma \left(2\left( \nu -m\right) +m\right) }{m!\Gamma \left( 2\left( \nu -m\right)+k\right) }\right) ^{1/2}
\\  \left( 1-z\overline{z}\right) ^{-m}\overline{z}^{m-k}P_{k}^{\left(m-k,2\left( \nu -m\right) -1\right) }\left( 1-2z\overline{z}\right) . 
\end{multline*}
Note that for the value $m=0$ associated with the ground state, the basis elements reduce to 
\begin{equation*}
\Phi _{k}^{(\nu ,0)}\left( z\right) =\left( 2\nu -1\right)^{1/2}\sqrt{\frac{\Gamma \left( 2\nu +k\right) }{\pi k!\Gamma(2\nu)}}z^{k}, \,\, k \geq 0,
\end{equation*}
and the $\mathcal{A}_{\nu ,0}^{2}$ is nothing else but the weighted Bergman space consisting of analytic functions $g$ in $\mathbb{D}$ such that
\begin{equation*}
\int\limits_{\mathbb{D}}\left| g\left( z\right) \right| ^{2}\mu_{\nu}(dz) \quad  < + \,\infty.
\end{equation*}
For that reason, the eigenspaces $\mathcal{A}_{\nu ,m}^{2}, m \geq 0$ were called in \cite{Mou} generalized Bergman spaces on the unit disk. 

In the sequel, we will be dealing with the following modification of the basis elements functions which takes into account the curvature of the Poincar\'e disc $\mathbb{D}_R$:

\begin{multline*}
\Phi _{j}^{(\nu, m, R)}\left( z\right) :=\left( \frac{\pi }{\left( 2\left( \nu R^{2}-m\right) -1\right)} 
\frac{\left( \max (m,j\right) )!\Gamma \left(2\left( \nu R^{2}-m\right) +\min (m,j\right) )}{\left( \min (m,j\right))!\Gamma \left(2\left( \nu R^{2}-m\right) +\max (m,j\right) )}\right) ^{-1/2}
\\
\frac{\left( -1\right) ^{\min \left( m,j\right) }}{\left(1-z\overline{z}R^{-2}\right) ^{m}}\left| \frac{z}{R}\right| ^{\left| m-j\right|}e^{-i\left( m-j\right) \arg z}P_{\min \left( m,j\right) }^{\left( \left|m-j\right| ,2\left( \nu R^{2}-m\right) -1\right) } 
\left( 1-2z\overline{z}R^{-2}\right).
\end{multline*}
Indeed, doing so enables us to recover analogous results in the flat (Euclidean) setting already derived in \cite{Dem-Mou}, \cite{Mou-Tou} via a geometrical contraction by letting $R \rightarrow \infty.$

\begin{remark}
The operator $\Delta _{\nu}$ can be mapped to the $\nu $-weight Maass Laplacian 
\begin{equation*}
y^{2}\left( \partial _{x}^{2}+\partial_{y}^{2}\right) - 2i\nu y\partial _{x}
\end{equation*}
on the Poincar\'{e} upper half-plane already studied in \cite{Com}. The physical meaning of the condition $\nu >1/2$ ensuring the existence of the discrete spectrum is that the magnetic field has to be strong enough to capture the particle in a closed orbit. In this case, the corresponding eigenfunctions are referred to as \textit{bound states} since a particle in such a state cannot leave it without an additional energy.
If this condition is not fulfilled, then the motion will be unbounded and the classical orbit of the particle will intercept the disk boundary whose points stands for $\left\{ \infty \right\}$ (\cite{Com}, p.189). 
\end{remark}

\section{Generalized negative binomial distribution: moment generating function and decomposition}
\noindent With the material introduced in the previous section, we are now ready to recall and analyze some probabilistic aspects of the generalized negative binomial distribution. In particular, we derive below its moment generating function whence we deduce its atomic decomposition. 

Let $\nu > 1/(2R^2)$ and $m=0,1,...,\left[\nu R^2 -1/2\right] $ be a fixed integer. Define the corresponding generalized NBCS by 
\begin{equation*}
\mid z,\nu ,m, R> := \left( \mathcal{N}_{\nu ,m}\left( z\right) \right) ^{-\frac{1}{2}}\sum\limits_{k=0}^{+\infty }\overline{\Phi _{k}^{(\nu, m, R)}\left( z\right)} \mid \phi_{k}>, \quad z \in \mathbb{D}_R,
\end{equation*}
where (\cite{Mou})
\begin{equation*}
\mathcal{N}_{\nu ,m, R}\left( z\right) =\pi ^{-1}(2\nu R^2 -2m-1)\left( 1-z\overline{z}\right) ^{-2\nu R^2},
\end{equation*}
is a normalizing factor. For any $j \geq 0$, the squared modulus of $\langle z, \nu, m, R | \phi_j\rangle$ gives the probability that $j$ photons are found in the state $\mid z,\nu ,m, R>$. This leads to the following definition:
\begin{definition}
\textit{Let }
The GNBD of parameters $(\nu, z, m, R)$ is defined by 
\begin{equation}\label{GNBD}
p_{j}^{(\nu, z, m, R)}:= \gamma _{j}^{(\nu, m, R)}\left( 1-\frac{|z|^2}{R^{2}}\right) ^{2\nu R^{2}-2m}\left( \frac{|z|^2}{R^{2}}\right) ^{\left| m-j\right| }
\left(P_{\frac{1}{2}\left( m+j-\left|m-j\right| \right) }^{\left( \left| m-j\right| ,2\nu R^{2}-2m-1\right)}\left( 1-\frac{2|z|^2}{R^{2}}\right) \right) ^{2}
\end{equation}
for $j \geq 0$, where 
\begin{equation}\label{Norm}
\gamma _{j}^{(\nu,m, R)}:=\frac{\Gamma \left( 1+\frac{1}{2}\left( m+j-\left|m-j\right| \right) \right) \Gamma \left(2\nu R^2-m+\frac{1}{2}\left(
\left| m-j\right| +j-m\right) \right) }{\Gamma \left( 1+\frac{1}{2}\left(m+j+|m-j|\right) \right) \Gamma \left( 2\nu R^2 -m-\frac{1}{2}(|m-j|+m-j)\right)}.
\end{equation}
\end{definition}
Note that 
\begin{equation*}
\gamma _{j}^{(\nu,m, R)} = \frac{j!\Gamma \left(2\nu R^2 -m\right) }{m!\Gamma\left(2\nu R^2 -2m+j\right)}
\end{equation*}
if $j \leq m$, while 
\begin{equation*}
\gamma _{j}^{(\nu,m, R)} = \frac{m!\Gamma\left(2\nu R^2 -2m+j\right)}{j!\Gamma \left(2\nu R^2 -m\right)}
\end{equation*}
otherwise. This simple observation together with another property satisfied by Jacobi polynomials are the main ingredients in the derivation of the moment generating function of the GNBD. 

\begin{proposition}
Let $|\xi| < 1$ be a complex number and write simply $\tau =  \tau(|z|, R) = |z|^2/R^2$ . Then the moment-generating function of the GNBD of parameters $(\nu,z, m, R)$ is given by
\begin{equation}\label{MGF}
G^{(\nu, \tau, m, R)}(\xi)=\left(\frac{1-\tau }{1-\tau \xi}\right) ^{2\nu R^2}\left( \frac{\left( \tau -\xi \right) \left( 1-\tau \xi \right) }{(1-\tau )^{2}}\right) ^{m}P_{m}^{\left(2\nu R^2-2m-1,0\right)} \left(1+ \frac{2\xi (1-\tau )^{2}}{\left( \tau -\xi \right) \left( 1-\tau \xi \right)}\right).  
\end{equation}
\end{proposition}

\begin{proof} From the very definition of the GNBD, we have
\begin{align}\label{MGS}
G^{(\nu, \tau, m, R)}(\xi) & = \sum_{j=0}^{+\infty }\xi ^{j}p_{j}^{(\nu, z, m, R)} \nonumber
\\& = \sum_{j=0}^{+\infty }\xi ^{j}\gamma _{j}^{(\nu, m, R)} \tau^{\left| m-j\right|}(1-\tau) ^{2\nu R^2-2m}
\left(P_{\frac{1}{2}\left( m+j-\left|m-j\right| \right) }^{\left( \left| m-j\right| ,2\nu R^2 -2m-1\right)}\left( 1- 2\tau \right) \right) ^{2}.
\end{align}
If $m=0$, then the statement of the proposition is already known. As a matter of fact, assume $m \geq 1$ and split the series \eqref{MGS} into two parts according to whether $\{j \leq m-1\}$ or $\{j \geq m\}$:
\begin{equation}
G^{(\nu, \tau, m, R)}(\xi)=G_1^{(\nu, \tau, m, R)}\left(\xi\right) + G_2^{(\nu,\tau, m, R)}\left(\xi\right),
\end{equation}
where 
\begin{align*}
G_1^{(\nu, \tau, m, R)}\left(\xi \right) &:= \frac{\Gamma \left(2\nu R^2-m\right) }{m!} \sum_{j=0}^{m-1}\frac{j!}{\Gamma\left(2\nu R^2-2m+j\right)}\left( 1-\tau \right) ^{2\nu R^2-2m}\tau ^{m-j}
\left(P_{j}^{\left(m-j,2\nu R^2 - 2m-1\right) }\left( 1-2\tau \right) \right) ^{2}\xi ^{j}
\\& -\frac{m!}{\Gamma \left(2\nu R^2-m\right)}\sum_{j=0}^{+\infty} \frac{\Gamma\left(2\nu R^2-2m+j\right)}{j!} \left( 1-\tau \right) ^{2\nu R^2-2m}\tau ^{j-m}\left(P_{m}^{\left(j-m,2\nu R^2-2m-1\right) }\left( 1-2\tau \right) \right) ^{2}\xi ^{j}, 
\end{align*}
and 
\begin{equation*}
G_2^{(\nu, \tau, m, R)}\left(\xi\right) := \frac{m!}{\Gamma \left(2\nu R^2-m\right)}\sum_{j=0}^{+\infty} 
\frac{\Gamma\left(2\nu R^2-2m+j\right)}{j!} \left(1-\tau \right) ^{2\nu R^2-2m}\tau ^{j-m}\left( P_{m}^{\left( j-m,2\nu R^2-2m-1\right) }\left(1-2\tau \right) \right) ^{2}\xi ^{j}.
\end{equation*}
Now, the expansion \eqref{DefJac} leads to the following transformation (see also \cite{AAR}, p.220)
\begin{equation*}
P_m^{(j-m, 2\nu R^2-2m-1)}(x) = \frac{j!\Gamma(2\nu R^2- m)}{m!\Gamma(2\nu R^2-2m+j)}\left(\frac{x-1}{2}\right)^{m-j}P_j^{(m-j, 2\nu R^2-2m-1)}(x)
\end{equation*}
which readily implies that $G_1^{(\nu, \tau, m, R)}(\xi) = 0$. Again, the same transformation yields 
\begin{equation*}
G_2^{(\nu, \tau, m, R)}\left(\xi\right) = \frac{\Gamma \left(2\nu R^2-m\right)}{m!}\left(1-\tau \right) ^{2\nu R^2-2m} 
\sum_{j=0}^{+\infty} \frac{j!}{\Gamma\left(2\nu R^2-2m+j\right)} \tau ^{m-j}\left( P_j^{(m-j,2\nu R^2 -2m-1)}\left(1-2\tau \right) \right) ^{2}\xi ^{j}.
\end{equation*}
Finally, recall the generating function (see eq. (3.5) in \cite{Sri-Rao})
\begin{multline*}
\sum_{j=0}^{+\infty }\frac{j!}{\left(1+\beta\right) _{j}} P_{j}^{(\gamma -j,\beta)}\left(x\right) P_{j}^{(\gamma -j,\beta)}\left( y\right) z^j= 
\left(1-z\right)^{\gamma }\left(1-\frac{\left(x-1\right) \left( y-1\right) z}{4}\right) ^{-1-\gamma -\beta}
\\ {}_2F_1\left(-\gamma, 1+\gamma +\beta,1+\beta; - \frac{\left(x+1\right) \left(y+1\right)z}{\left(1-z\right) \left(4-\left(x-1\right) \left(y-1\right)z\right) }\right)   
\end{multline*}
which converges absolutely at least in a small open disc centered at the origin. Since $\tau > 0$ is fixed, then we can choose non real $\xi$ with small enough modulus and use the above generating function together with \eqref{RepJac} in order to derive the desired expression. But,
\begin{equation*}
\xi \mapsto \left( \frac{\left( \tau -\xi \right) \left( 1-\tau \xi \right) }{(1-\tau )^{2}}\right) ^{m}P_{m}^{\left(2\nu-2m-1,0\right)} \left(1+ \frac{2\xi (1-\tau )^{2}}{\left( \tau -\xi \right) \left( 1-\tau \xi \right)}\right)
\end{equation*}
is well defined and analytic in the open unit disc. Hence, \eqref{MGF} extends analytically to $\xi \in \mathbb{D}$.   
\end{proof}

\begin{remark}
$G^{(\nu, \tau, m, R)}$ is a continuous function in the closed unit disc and as such, \eqref{MGF} remains valid on the unit circle $\{|\xi| = 1\}$. 
\end{remark}
\begin{remark}[Contraction principle]
Using the expansion \eqref{DefJac}, we get
\begin{multline*}
\left( \frac{\left( \tau -\xi \right) \left( 1-\tau \xi \right) }{(1-\tau )^{2}}\right) ^{m}P_{m}^{\left(2\nu-2m-1,0\right)} \left(1+ \frac{2\xi (1-\tau )^{2}}{\left( \tau -\xi \right) \left( 1-\tau \xi \right)}\right)
 = \xi^m \sum_{j=0}^m\binom{2\nu R^2-m-1}{j}\binom{m}{j}  \\ \left(\frac{\tau}{(1-\tau)^2}\right)^j\left(\frac{(1-\xi)^2}{\xi}\right)^j.
 \end{multline*}
 Letting $R \rightarrow \infty$ for fixed $|z|$ then $\tau \rightarrow 0$ and the right-hand side of \eqref{MGF} tends to 
\begin{equation*}
e^{2\nu |z|^2 (\xi-1)}\xi^m \sum_{j=0}^m\frac{1}{j!}\binom{m}{j} \left(2\nu\frac{(1-\xi)^2}{\xi}\right)^j = e^{2\nu |z|^2 (\xi-1)\tau}\xi^m L_m^{(0)}\left(-2\nu|z|^2\frac{(1-\xi)^2}{\xi}\right),
\end{equation*}
which is the moment generating function of the generalized Poisson distribution of parameter $\lambda = 2\nu|z|^2$ (\cite{Dem-Mou}). This limiting result is not surprising and is rather expected since the curvature of the Poincar\'e disc $\mathbb{D}_R$ of radius $R$ tends to zero. 
\end{remark}

\subsection{Atomic decomposition}
From now on, we will assume $R=1$ and delete the superscript $R$ from our previous notations. In this paragraph, we decompose the GNBD attached with a given hyperbolic Landau level $m \geq 1$ as a perturbation of the NBD ($m=0$) by a finitely-supported (signed) measure. To proceed, we appeal to the last remark which shows that 
\begin{equation*}
G^{(\nu, \tau, m)}(\xi) = \left(\frac{1-\tau }{1-\tau \xi }\right) ^{2\nu}\xi^m \sum_{j=0}^m\binom{2\nu-m-1}{j}\binom{m}{j}  \left(\frac{\tau}{(1-\tau)^2}\right)^j\left(\frac{(1-\xi)^2}{\xi}\right)^j.
\end{equation*}
Consequently, the Fourier transform of the GNBD of parameters $(\nu,m)$ reads 
\begin{equation*}
G^{(\nu, \tau, m)}(e^{iu})=\left(\frac{1-\tau }{1-\tau e^{iu}}\right) ^{2\nu}e^{ium} \sum_{j=0}^m\binom{2\nu-m-1}{j}\binom{m}{j} \left(-\frac{4\tau}{(1-\tau)^2}\right)^j \sin^{2j}(u/2).
\end{equation*}
In the right-hand side of the last equation, the two first factors are the Fourier transforms of a NBD of parameters $(\nu,\tau)$ and of a Dirac mass at $m$. As to the remaining sum finite sum, it is the Fourier transform of a finitely-supported signed atomic measure. Indeed, if $m \geq 1$ then the linearization formula (\cite{Gra-Ryz}, p.31)
\begin{equation}\label{Linea}
4^j\sin ^{2j}(u/2)=\binom{2j}{j} + 2\sum_{k=1}^{j}\left(-1\right)^{k}\binom{2j}{j-k} \cos(ku){\bf 1}_{\{j \geq 1\}}.
\end{equation}
yields
\begin{align*}
\sum_{j=0}^m\binom{2\nu-m-1}{j}\binom{m}{j} &\left(-\frac{4\tau}{(1-\kappa)^2}\right)^j \sin^{2j}(u/2) = 1+ 
\\& 2\sum_{k=1}^m(-1)^k\left\{\sum_{j=k}^m \binom{2\nu-m-1}{j}\binom{m}{j} \binom{2j}{j-k} \left(-\frac{\tau}{(1-\tau)^2}\right)^j\right\}\cos(ku)
\end{align*}
which is the Fourier transform of  
\begin{align*}
\delta_0 + \sum_{k=1}^m(-1)^k\left\{\sum_{j=k}^m \binom{2\nu-m-1}{j}\binom{m}{j} \binom{2j}{j-k} \left(-\frac{\tau}{(1-\tau)^2}\right)^j\right\}[\delta_k + \delta_{-k}].
\end{align*}
Set 
\begin{equation*}
Q_{k}^{(\nu,m)}(x) := (-1)^k\sum_{j=k}^m \binom{2\nu-m-1}{j}\binom{m}{j} \binom{2j}{j-k}(-x)^j, \quad 1 \leq k \leq m, \, \, Q_{0}^{(\nu,m)}(x) := 1
\end{equation*}
then we have proved that
\begin{proposition}\label{AD}
The GNBD of parameters $(\nu, \tau, m)$ admits the following decomposition: 
\begin{multline*}
\left(\sum_{j \geq 0} p_{j}^{(\nu,\tau, 0)}\delta_j\right) \star \delta_m \star \left\{\delta_0+\sum_{k=1}^mQ_{k}^{(\nu,m)}\left(\frac{\tau}{(1-\tau)^2}\right)[\delta_{k} + \delta_{-k}]\right\} = 
 \left(\sum_{j \geq 0} p_{j}^{(\nu,\tau, 0)}\delta_j\right) \\ \star \sum_{k=0}^{2m}Q_{|k-m|}^{(\nu,m)}\left(\frac{\tau}{(1-\tau)^2}\right)\delta_k.
\end{multline*}
\end{proposition}

\begin{remark}
Using the Legendre duplication formula \eqref{Leg}, the polynomials $Q_{k,m}^{(\nu)}$ may be expressed through a ${}_3F_2$ hypergeometric polynomial: 
\begin{align*}
Q_{k}^{(\nu,m)}(x) &= x^k\sum_{j=0}^{m-k} \binom{2\nu-m-1}{j+k}\binom{m}{j+k} \binom{2j+2k}{j}(-x)^j
\\& = \frac{m!(4x)^k}{\sqrt{\pi}}\sum_{j=0}^{m-k} \frac{\Gamma(2\nu-m)}{\Gamma(j+k+1)\Gamma(2\nu-m-k-j)}\frac{\Gamma(j+k+1/2)(-4x)^j}{\Gamma(j+2k+1)(m-k-j)!j!}
\\& = \binom{m}{k}(m-2\nu)_k (-x)^k\sum_{j=0}^{m-k} (m+k-2\nu)_j \frac{(k-m)_j(k+1/2)_j}{(k+1)_j(2k+1)_j} \frac{(-x)^j}{j!}
\\& = \binom{m}{k}(m-2\nu)_k \,(-x)^k \, {}_3F_2(k-m, m+k-2\nu, k+(1/2), k+1, 2k+1; -x).
\end{align*}
\end{remark}

\subsection{Photon counting statistics}
To define a measure of non classicality of a quantum state, one can follow several different approach. An earlier attempt was initiated by Mandel \cite{Man} who investigated radiation fields and introduced for any random variable the parameter
\begin{equation*}
Q = Q(X) = \frac{\mathbb{V}ar(X)}{\mathbb{E}(X)}-1
\end{equation*}
to measure the deviation of the photon number statistics from the Poisson distribution for which $Q=0$. If $Q < 0$, then the underlying statistics are said to be \textit{sub-Poissonian}
statistics and describes the anti-bunching of the light which reveals the quantum nature of light. Super-Poisson statistics corresponds rather to $Q > 0$ and the bunching phenomenon occurs. In our setting, the random variable $X_{0}$ associated with the lowest hyperbolic Landau level $m=0$ obeys the NBD and it is already known that the photon counting statistics are in this case super-Poissonian (\cite{GJT}). As a matter of fact, the random variable $X_0$ describes a thermal light field with bunched photon spacing and a larger number of fluctuations than a Poissonian (coherent) state. 
For higher hyperbolic Landau levels $\epsilon_m, m \geq 1$, the situation is very different since the proposition below shows that the three regimes (anti-bunching, Poissonian, bunching) are possible according to the values of the light intensity $\tau$. In particular, an anti-bunching region emerges for small enough $\tau$ 
while the Poissonian regime corresponds to a circle in the Poincar\'e disc.
\begin{proposition}
For any $\nu > 1/2$ and any $m=0,1,2,...,\left\lfloor \nu-(1/2)\right\rfloor$, there exists a non negative real number $\rho = \rho(\nu,m)$ such that the photon counting statistics are 
\begin{itemize}
\item sub-Poissonian $z\in \mathbb{D}(0,\rho)$. The corresponding states $\mid z,\nu ,m,R>$ are non-classical (anti-bunching).
\item Poissonian for $z \in \partial \mathbb{D}(0,\rho)$. Here, $\mid z,\nu ,m,R>$ becomes pure coherent.
\item Super-Poissonian for $z \notin \overline{\mathbb{D}(0,\rho)}$. For such complex numbers, the states $\mid z,\nu ,m,R>$ describe a thermal light (bunching).
\end{itemize}
\end{proposition}

\begin{proof} 
From the expression we obtained for the moment generating function, we readily derive
\begin{equation*}
\mathbb{E}(X) = \frac{2\tau \nu}{1-\tau }+m
\end{equation*}
and
\begin{equation*}
\mathbb{V}ar(X) = \frac{2\tau}{(1-\tau )^{2}}(\nu+ m(2\nu-m-1)).
\end{equation*}
As a result, the Mandel parameter is given by 
\begin{equation}
Q=\frac{2\tau(\nu + m(2\nu-m-1)}{(1-\tau)(2\tau\nu + m(1-\tau))} - 1 = \frac{(2\nu-m)\tau^2 + 2m(2\nu-m)\tau-m}{(1-\tau)(2\tau\nu + m(1-\tau))}.
\end{equation}
Since $\tau = |z|^2 \in (0,1)$, then we only need to study the sign of the numerator of $Q$ with respect to the variable $\tau$. Note that the product of the roots is 
\begin{equation*}
-\frac{m}{2\nu-m} \in (-1,0]
\end{equation*}
so that there is unique positive root $\rho(\nu,m)$ lying in $[0,1)$: 
\begin{equation*}
\rho(\nu,m) = \frac{\sqrt{m^2(2\nu-m)^2 + m(2\nu-m)} - m(2\nu-m)}{2\nu-m} = \sqrt{m^2 + \frac{m}{2\nu-m}} - m.
\end{equation*}
 The three assertions of the proposition then follow from straightforward computations. 
\end{proof}
\begin{remark}
Considering $\rho(\nu,m)$ as a function of a real variable $m \in [0, \nu-1/2]$, then its derivative is given by 
\begin{equation*}
\left(m^2 + \frac{m}{2\nu-m}\right)^{-1/2}\left[m + \frac{\nu}{(2\nu-m)^2} - \left(m^2 + \frac{m}{2\nu-m}\right)^{1/2}\right]. 
\end{equation*}
But $2\nu/(2\nu-m) \geq 1$ therefore the term betwenn brackets is positive. It follows that the radius of the anti-bunching region $\rho(\nu,m)$ stretches as $m$ approaches $[\nu-(1/2)]$. Figure 1 below shows the increase of $m \mapsto \rho^2(11/2,m)$:
\begin{figure}[H]
\centering
\includegraphics[width=0.50\textwidth]{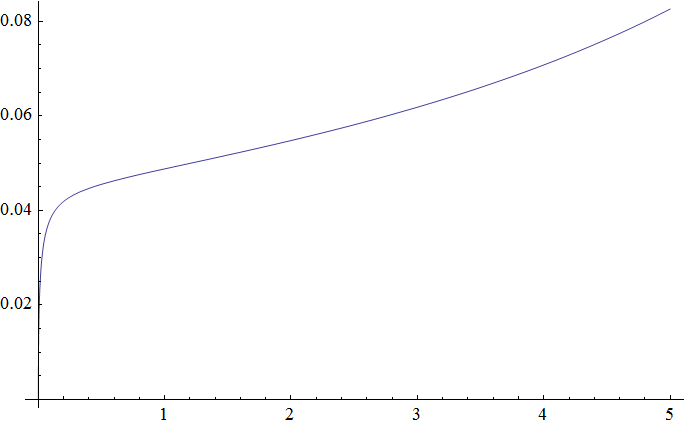}
\caption{{Square of the Anti-bunching radius $\rho(11/2,m)$}}
\end{figure}
\end{remark}

\section{L\'evy-Kintchine representation and Quasi-Infinite divisibility}
Recall that a random variable $Y$ is infinitely divisible if for any $n \geq 2$, there exist $n$ independent and identically distributed random variables $(Y_k^{(n)})_{k=1}^n$ such that (\cite{Sato}, \cite{Ste-VH})
\begin{equation*}
Y = \sum_{k=1}^nY_k^{(n)}.
\end{equation*}
For this property to hold, it is necessary that the characteristic function of $Y$ does not vanish (\cite{Ste-VH}, CH.IV, Proposition 2.4). In order to check whether this property holds or not for the GNBD, we first derive from \eqref{MGF} its characteristic function 
\begin{equation*}
G^{(\nu,\tau, m)}\left(e^{iu}\right) = \left(\frac{1-\tau }{1-\tau e^{iu} }\right) ^{2\nu}e^{imu} \left( \frac{\left(\tau e^{-iu} -1 \right) \left( 1-\tau e^{iu} \right) }{(1-\tau )^{2}}\right) ^{m}P_{m}^{\left(2\nu-2m-1,0\right)} \left(1+ \frac{2(1-\tau )^{2}}{\left(\tau e^{-iu} -1 \right) \left( 1-\tau e^{iu} \right)}\right)
\end{equation*}
for any $u \in \mathbb{R}$ then recall that 
\begin{equation*}
u \mapsto  \left(\frac{1-\tau }{1-\tau e^{iu} }\right) ^{2\nu}
\end{equation*}
is the characteristic function of the NBD which is already known to be infinitely divisible (\cite{Sato}, \cite{Ste-VH}). Moreover, 
\begin{equation}\label{InfiDiv}
\left(\frac{1-\tau }{1-\tau e^{iu} }\right) ^{2\nu} = \exp\left\{\int \left(e^{iux}-1\right) \sum_{j \geq 1}\frac{\tau^j}{j}\delta_j(dx)\right\}.
\end{equation}
As to the product 
\begin{equation*}
\left( \frac{\left(\tau e^{-iu} -1 \right) \left( 1-\tau e^{iu} \right) }{(1-\tau )^{2}}\right)^{m}P_{m}^{\left(2\nu-2m-1,0\right)} \left(1+ \frac{2(1-\tau )^{2}}{\left(\tau e^{-iu} -1 \right) \left( 1-\tau e^{iu}\right)}\right), 
\end{equation*}
we use \eqref{Fact} in order to write it as 
\begin{equation*}
\frac{(2\nu-m)_m}{2^{m}m!} \left( \frac{\left(\tau e^{-iu} -1 \right) \left( 1-\tau e^{iu} \right) }{(1-\tau )^{2}}\right)^{m} \prod_{i=1}^m \left(1+ \frac{2(1-\tau )^{2}}{(\tau e^{-iu} -1)(1-\tau e^{iu})} - x_i^{(m)}\right)
\end{equation*}
where $-1 < x_1^{(m)} < \dots < x_m^{(m)} < 1$ are the simple zeros of $P_{m}^{\left(2\nu-2m-1,0\right)}$. Equivalently, 
\begin{multline}\label{A0}
\left(\frac{\left(\tau e^{-iu} -1 \right) \left( 1-\tau e^{iu} \right) }{(1-\tau )^{2}}\right)^{m}P_{m}^{\left(2\nu-2m-1,0\right)} \left(1+ \frac{2(1-\tau )^{2}}{\left(\tau e^{-iu} -1 \right) \left( 1-\tau e^{iu}\right)}\right) = \frac{(2\nu-m)_m}{m!} 
\\  \prod_{i=1}^m \left(1 -\frac{1-x_i^{(m)}}{2}\frac{(1+\tau^2-2\tau\cos u)}{(1-\tau)^2} \right). 
\end{multline}
Since 
\begin{equation*}
\frac{(1+\tau^2-2\tau\cos u)}{(1-\tau)^2} \geq 1
\end{equation*}
and $(1-x_i^{(m})/2 \in (0,1)$ for any $1 \leq i \leq m$, then the characteristic function of the GNBD does not vanish for any $u \in \mathbb{R}$ provided that 
\begin{equation}\label{Equa1}
\frac{(1+\tau^2-2\tau\cos u)}{(1-\tau)^2} < \frac{2}{(1-x_1^{(m)})}. 
\end{equation}
But 
\begin{equation*}
\max_{u \in [-1,1]} \frac{(1+\tau^2-2\tau\cos u)}{(1-\tau)^2} = \frac{(1+\tau)^2}{(1-\tau)^2}
\end{equation*}
therefore \eqref{Equa1} is equivalent to
\begin{equation}\label{Equa2}
\tau < \frac{\sqrt{2} - \sqrt{1-x_1^{(m)}}}{\sqrt{2} + \sqrt{1-x_1^{(m)}}}.
\end{equation}
Assuming \eqref{Equa2} holds, we prove the following:  

\begin{proposition}\label{prop4}
There exist signed measures $\mu_n^{(\nu, \tau, m)}, 1 \leq n \leq m$ satisfying 
\begin{equation}\label{QuasiInfi}
\mu_n^{(\nu,\tau,m)}\{0\} = 0, \quad \int (1 \wedge x^2) \mu_n^{(\nu,\tau,m)}(dx) \, < \, \infty,
\end{equation}
and such that 
\begin{equation*}
\exp\left\{\int \left(e^{iux}-1\right) \sum_{n=1}^m\mu_k(dx)\right\} = \left( \frac{\left(\tau e^{-iu} -1 \right) \left( 1-\tau e^{iu} \right) }{(1-\tau )^{2}}\right)^{m}P_{m}^{\left(2\nu-2m-1,0\right)} \left(1+ \frac{2(1-\tau )^{2}}{\left(\tau e^{-iu} -1 \right) \left( 1-\tau e^{iu}\right)}\right). 
\end{equation*}
\end{proposition}

\begin{proof}
If \eqref{Equa2} holds then 
\begin{equation*}
\frac{1-x_n^{(m)}}{2}\frac{(1+\tau^2-2\tau\cos u)}{(1-\tau)^2} < 1
\end{equation*}
for any $1 \leq n \leq m$ so that we can expand  
\begin{align*}
\ln\left(1- \frac{1-x_n^{(m)}}{2}\frac{(1+\tau^2-2\tau\cos u)}{(1-\tau)^2}\right) &= - \sum_{j \geq 1} \frac{1}{j}\left(\frac{1-x_n^{(m)}}{2(1-\tau)^2}\right)^j \left[(1-\tau)^2 + 4\tau\sin^2(u/2)\right]^j
\\& = - \sum_{j \geq 1} \frac{1}{j}\left(\frac{1-x_n^{(m)}}{2}\right)^j \sum_{k=0}^j\binom{j}{k}\left(\frac{4\tau}{(1-\tau)^2}\right)^k \sin^{2k}(u/2).
\end{align*}
The term corresponding to $k=0$ in the last series is 
\begin{align*}
-\sum_{j \geq 1} \frac{1}{j}\left(\frac{1-x_n^{(m)}}{2}\right)^j = \ln\left(\frac{1+x_n^{(m)}}{2}\right)
\end{align*}
while
\begin{align*}
\sum_{j \geq 1} \frac{1}{j}\left(\frac{1-x_n^{(m)}}{2}\right)^j  \sum_{k=1}^j\binom{j}{k}\left(\frac{4\tau}{(1-\tau)^2}\right)^k &\sin^{2k}\left(\frac{u}{2}\right)
 = \sum_{k \geq 1}\left(\frac{4\tau}{(1-\tau)^2}\right)^k \sin^{2k}\left(\frac{u}{2}\right) \sum_{j \geq k} \frac{1}{j}\binom{j}{k}\left(\frac{1-x_n^{(m)}}{2}\right)^j 
\\&  =  \sum_{k \geq 1}\frac{1}{k!}\left(\frac{2\tau(1-x_n^{(m)})}{(1-\tau)^2}\right)^k \sin^{2k}\left(\frac{u}{2}\right) \sum_{j \geq 0} \frac{\Gamma(j+k)}{j!}\left(\frac{1-x_n^{(m)}}{2}\right)^j 
\\& =   \sum_{k \geq 1}\frac{1}{k}\left(\frac{2\tau(1-x_n^{(m)})}{(1-\tau)^2}\right)^k \sin^{2k}\left(\frac{u}{2}\right)\sum_{j \geq 0} \frac{(k)_j}{j!}\left(\frac{1-x_n^{(m)}}{2}\right)^j 
\\& = \sum_{k \geq 1}\frac{1}{k}\left(\frac{\tau(1-x_n^{(m)})}{(1+x_n^{(m)})(1-\tau)^2}\right)^k 4^k\sin^{2k}\left(\frac{u}{2}\right).
\end{align*}
Using the linearization formula \eqref{Linea}, we get 
\begin{multline*}
\ln\left(1- \frac{1-x_n^{(m)}}{2}\frac{(1+\tau^2-2\tau\cos u)}{(1-\tau)^2}\right) = \ln\left(\frac{1+x_n^{(m)}}{2}\right) - \sum_{k \geq 1}\frac{1}{k}\left(\frac{\tau(1-x_n^{(m)})}{(1+x_n^{(m)})(1-\tau)^2}\right)^k
\\ \int e^{iux}\left\{\binom{2k}{k}\delta_0(dx) + \sum_{s=1}^{k}\left(-1\right)^{s}\binom{2k}{k-s} [\delta_{s} + \delta_{-s}](dx)\right\}.
\end{multline*}
Set 
\begin{equation*}
A_n^{(\tau, m)} := \frac{\tau(1-x_n^{(m)})}{(1+x_n^{(m)})(1-\tau)^2}
\end{equation*}
and note that \eqref{Equa2} entails $A_n^{(\tau,m)} \in (0,1/4)$ for any $1 \leq n \leq m$. Then, the computations performed in the proof of Proposition 3 in \cite{Dem-Mou} lead to the following expressions: 
\begin{equation}\label{A1}
- \sum_{k \geq 1}\frac{1}{k}\binom{2k}{k}\left(A_n^{(\tau, m)}\right)^k = 2\ln\left(\frac{1+\sqrt{1-4A_n^{(\tau, m)}}}{2}\right), 
\end{equation}
and for any $s \geq 1$,
\begin{equation}\label{A2} 
- \sum_{k \geq s}\frac{1}{k}\binom{2k}{k-s} \left(A_n^{(\tau, m)}\right)^k = \frac{1}{s}\left[\alpha\left(4A_n^{(\tau, m)}\right)\right]^s,
\end{equation}
where 
\begin{equation*}
\alpha(x) := \frac{x}{(1+\sqrt{1-x})^2}. 
\end{equation*}
Set 
\begin{equation*}
\mu_n^{(\nu, \tau, m)} := -\sum_{s \geq 1}\frac{(-1)^s}{s}\left[\alpha\left(4A_n^{(\tau, m)}\right)\right]^s [\delta_s+\delta_{-s}], \quad 1 \leq n \leq m,
\end{equation*}
then  
\begin{equation}\label{A3}
\int \mu_n^{(\tau,m)}(dx) = 2\ln\left[1+\alpha(4A_n^{(\tau,m)})\right] = -2\ln\left(\frac{1+\sqrt{1-4A_n^{(\tau, m)}}}{2}\right). 
\end{equation}
Gathering \eqref{A1}, \eqref{A2} and \eqref{A3}, we obtain 
\begin{equation*}
\ln\left(1- \frac{1-x_n^{(m)}}{2}\frac{(1+\tau^2-2\tau\cos u)}{(1-\tau)^2}\right) = \ln\left(\frac{1+x_n^{(m)}}{2}\right) + \int \left(e^{iux}-1\right)\mu_n^{(\nu,\tau,m)}(dx).
\end{equation*}
Finally, summing the last expression over $k \in \{1,\cdots, m\}$ and keeping in mind \eqref{A0}, it remains to prove that 
\begin{equation}\label{prod}
\frac{(2\nu-m)_m}{m!}\prod_{n=1}^m\frac{1+x_n^{(m)}}{2} = 1. 
\end{equation}
To this end, we appeal to the relation $P_m^{(a,b)}(x) = (-1)^mP_m^{(b,a)}(-x)$ together with the representation \eqref{RepJac} to see that
\begin{equation*}
0 = P_m^{(2\nu-2m-1,0)}(x_n^{(m)}) = (-1)^m P_m^{(0,2\nu-2m-1)}(-x_n^{(m)}) = (-1)^m{}_2F_1\left(-m, 2\nu-m, 1, \frac{1+x_n^{(m)}}{2}\right),
\end{equation*}
for any $1 \leq n \leq m$. Consequently, $\{(1+x_n^{(m)})/2, \, 1 \leq n \leq m\}$ is the set of zeros of the monic hypergeometric polynomial
\begin{equation*}
x \mapsto (-1)^m\frac{m!}{(2\nu-m)_m}{}_2F_1\left(-m, 2\nu-m, 1,x \right) = (-1)^m\frac{m!}{(2\nu-m)_m}\sum_{j=0}^m \binom{m}{j} \frac{(2\nu-m)_j}{(1)_j}(-x)^j 
\end{equation*}
therefore 
\begin{equation*}
\prod_{n=1}^m\frac{1+x_n^{(m)}}{2}  = \frac{m!}{(2\nu-m)_m}
\end{equation*}
as required. Since the measure $\mu_n^{(\nu, \tau,m)}$ clearly satisfies \eqref{QuasiInfi}, the proposition is proved. 
\end{proof}
\begin{remark}
A signed measure satisfying \eqref{QuasiInfi} is called a quasi-L\'evy measure (see e.g. E.12.2 and E.12.3 in \cite{Sato}). 
\end{remark}

Combining the previous proposition and \eqref{InfiDiv}, we deduce the following L\'evy-Khintchine type representation for the NGBD:
\begin{corollary}
Assume \eqref{Equa2} holds. Then 
\begin{equation*}
G^{(\nu, \tau, m)}(e^{iu}) = \exp\left\{imu + \int \left(e^{iux}-1\right)\left\{\sum_{j \geq 1}\frac{\tau^j}{j}\delta_j(dx) + \sum_{n=1}^m\mu_n^{(\nu, \tau,m)}(dx)\right\}\right\}
\end{equation*}
for any $u \in \mathbb{R}$. 
\end{corollary}

\section{A new infinitely-divisible distribution}
We have already seen that the GNBD associated with a given hyperbolic Landau level $\epsilon_m$ is not infinitely-divisible unless $m=0$. Nonetheless, when $m \geq 1$, we can consider the total variation of the quasi-L\'evy measure
\begin{equation*}
\sum_{n=1}^m\mu_n^{(\nu, \tau,m)}
\end{equation*}
in order to obtain a L\'evy measure and as such an infinitely divisible distribution (or equivalently a L\'evy processes, \cite{Sato}). More precisely, 
\begin{equation*}
\sum_{j \geq 1}\frac{\tau^j}{j}\delta_j  + \sum_{n=1}^m|\mu_n^{(\nu, \tau,m)}| = \sum_{j \geq 1}\frac{\tau^j}{j}\delta_j + \sum_{n=1}^m \sum_{s \geq 1}\frac{1}{s}\left[\alpha\left(4A_n^{(\tau, m)}\right)\right]^s [\delta_s+\delta_{-s}]
\end{equation*}
is a L\'evy measure and the characteristic function of the corresponding infinitely-divisible distribution follows from the lines of the proof of Proposition \ref{prop4} read backward. There, we should replace $\sin^{2k}(u/2)$ by $\cos^{2k}(u/2)$ and use the linearization formula 
\begin{equation*}
4^j\cos^{2j}(u/2)=\binom{2j}{j} + 2\sum_{k=1}^{j}\binom{2j}{j-k} \cos(ku){\bf 1}_{\{j \geq 1\}}
\end{equation*}
together with
\begin{equation*}
(1-\tau)^2 + 4\tau \cos^2(u/2) = 1+\tau^2 + 2\tau \cos(u). 
\end{equation*}
The issue of the computations is: 
\begin{proposition}
Assume \eqref{Equa2} holds. Then 
\begin{multline*}
\exp\left\{imu + \int \left(e^{iux}-1\right)\left\{\sum_{j \geq 1}\frac{\tau^j}{j}\delta_j(dx) + \sum_{n=1}^m|\mu_n^{(\nu, \tau, m)}|(dx)\right\}\right\} = \left(\frac{1-\tau }{1-\tau e^{iu} }\right) ^{2\nu} e^{imu}\frac{m!}{(2\nu-m)_m}
\\ \prod_{n=1}^m \left(1 -\frac{1-x_n^{(m)}}{2}\frac{(1+\tau^2+2\tau\cos u)}{(1-\tau)^2} \right)^{-1}
\end{multline*}
is the characteristic function of an infinitely-divisible probability distribution.
\end{proposition}
Apart from the drift part $u \mapsto imu$, the remaining integral 
\begin{equation*}
u \mapsto \int \left(e^{iux}-1\right)\left\{\sum_{j \geq 1}\frac{\tau^j}{j}\delta_j(dx) + \sum_{n=1}^m|\mu_n^{(\nu, \tau, m)}|(dx)\right\}
\end{equation*}
is the characteristic exponent of a compound L\'evy process (see e.g. \cite{Sato}, p.18-19) whose intensity is given by 
\begin{equation*}
\lambda := \int \left\{\sum_{j \geq 1}\frac{\tau^j}{j}\delta_j(dx) + \sum_{n=1}^m|\mu_n^{(\nu, \tau, m)}|(dx)\right\} = -\ln(1-\tau) - 2\sum_{n=1}^m \ln\left[1-\alpha\left(4A_n^{(\tau, m)}\right)\right] > 0
\end{equation*}
and jump distribution is the probability measure 
\begin{equation*}
\frac{1}{\lambda}\left\{\sum_{j \geq 1}\frac{\tau^j}{j}\delta_j+ \sum_{n=1}^m|\mu_n^{(\nu, \tau, m)}|\right\}.
\end{equation*}

\section{Concluding remarks}
In this paper, we performed the analysis of probability distributions arising from the coherent states formalism applied to orthonormal bases of generalized Bergman spaces. Surprisingly, the moment generating function takes a product form and we noticed that the failure of the infinite-divisibility property for these distributions goes in parallel with the appearance of the anti-bunching regions for their counting statistics. We also introduced a new infinitely-divisible probability distribution and it is clear from the unboundedness of its jumps in both directions that the L\'evy process it generates is valued in $\mathbb{Z}$. In this respect, it would be interesting to relate the latter as well to a quantum model. Finally, other interesting connections between the Maass operator $\Delta_{\nu}$ and stochastic processes theory have been already investigated in \cite{Ike-Mat}.

{\bf Acknowledgment}: we would like to thank Alain Comtet for his helpful remarks and the fruitful discussions on the relevance of the hyperbolic Landau levels in physics.




\begin{thebibliography}{99}
\bibitem{AAR}\emph{G. E. Andrews, R. Askey, R. Roy}. Special functions. {\it Cambridge University Press}. 1999.
\bibitem{Bar}\emph{S. M. Barnett}. Negative binomial states of quantized radiation fields. {\it J. Mod. Opt.} {\bf 45}, no.10, 1998. 2201-2205.
\bibitem{Com}\emph{A. Comtet}. On the Landau levels on the hyperbolic plane. {\it Ann. Physics}. {\bf 173} (1987), no. 1, 185-209. 
\bibitem{Dem-Mou}\emph{N. Demni, Z. Mouayn}. Analysis of generalized Poisson distributions associated with higher Landau levels. {\it Infin. Dim. Anal. Quantum Probab. Rel. Topics.} {\bf 18}, No. 4 (2015).  
\bibitem{Fu-Sas}\emph{H. C. Fu, R. Sasaki}. Negative binomial and multinomial states: probability distributions and coherent states. {\it J. Math. Phys}. {\bf 38}, (1997), no. 8, 3968-3987. 
\bibitem{Gaz}\emph{J. P. Gazeau}. Coherent States in Quantum Physics. {\it Wiley, Weinheim} (2009).
\bibitem{Gio}\emph{M. Giovannini}. Multiplicity distributions in gravitational and strong interactions. {\it Physics Letters B}. {\bf 691}, Issue 5. (2010), 274-278. 
\bibitem{GJT}\emph{T. Gantsog, A. Joshi and R. Tanas}. Phase properties of binomial and negative binomial states. {\it Quantum Optics: Journal of the European Optical Society Part B}, Volume 6, Number 6.
\bibitem{Gra-Ryz}\emph{I. S. Gradshteyn, I. M. Ryzhik}. Table of Integrals, Series and Products. {\it Academic Press, INC, Seven-th Edition}. 2007.
\bibitem{Ike-Mat}\emph{N. Ikeda, H. Matsumoto}. Brownian motion on the hyperbolic plane and Selberg trace formula. {\it J. Func. Anal}. {\bf 163}, 1999. 63-110.
\bibitem{Man}\emph{L. Mandel}. Sub-Poissonian photon statistics in resonance fluorescence. {\it Opt. Lett}. {\bf 4}, no. 7. (1979). 205-207.  
\bibitem{Mo}\emph{Z. Mouayn}. Coherent states attached to Landau levels on the Poincar\'{e} disc, \textit{J. Phys. A: Math. Gen}. {\bf 38}, (2005) 9309-9316. 
\bibitem{Mou}\emph{Z. Mouayn}. Husimi's Q-function of the isotonic oscillator in a generalized Binomial states representation. {\it Math. Phys. Anal. Geom}. {\bf 17}, (2014), no.3-4. 289-303.
\bibitem{Mou-Tou} \emph{Z. Mouayn, A. Touhami}. Probability distributions attached to generalized Bargmann-Fock spaces in the complex plane. {\it Inf. Dimens. Anal. Quantum Probab. Related fields}, \textbf{13},
No.2, 2010.
\bibitem{Sri-Rao} \emph{A. Srivastava, A. B. Rao}. On a polynomial of the form $F_4$. {\it Indian J. Pure Appl. Math}. {\bf 6} (1975), no. 11, 1326-1334.
\bibitem{Sato}\emph{K. Sato}. L\'evy processes and infinitely divisible distributions. {\it Cambridge Studies in Advanced Mathematics}. 1999.
\bibitem{Ste-VH} \emph{F.W. Steutel, K. Van Harn}. Infinite divisibility of probability distributions on the real line. {\it Marcel Dekker, Inc. New York. Basel}. 2004.
\bibitem{SST}\emph{D. Stoler, B. E. A. Saleh, M. C. Teich.} Binomial states of the quantized radiation field. {\it Opt. Acta}. {\bf 32}, no.3, 1985. 345-355. 
\end{thebibliography}
\end{document}